\documentclass[12pt]{amsart}
\usepackage{color,amscd,amsmath,amssymb,enumerate,stmaryrd}
\usepackage[all]{xy}   
\usepackage{charter}
\usepackage[unicode,colorlinks=true,linkcolor=black,citecolor=black]{hyperref}

\DeclareMathOperator{\Coker}{Coker}
\newcommand{\Comod}[1]{\mathbf{Comod}(#1)}
\newcommand{\Comodf}[1]{\mathbf{Comod}_{\mathbf f}(#1)}
\newcommand{\algebras}[1]{#1\text{-}\mathbf{Alg}}

\DeclareMathOperator{\Grp}{\bf{Grp}}
\DeclareMathOperator{\Hom}{Hom}
\DeclareMathOperator{\Img}{Im}

\DeclareMathOperator{\Ker}{Ker}

\DeclareMathOperator{\Spec}{Spec}

\DeclareMathOperator{\Tor}{Tor}
\DeclareMathOperator{\tf}{tf}

\newcommand{\oti}{\otimes}

\newcommand{\afr}{\mathfrak a}
\newcommand{\hh}{\mathfrak h}
\newcommand{\lfr}{\mathfrak l}
\newcommand{\mm}{\mathfrak m}

\newcommand{\pp}{\mathfrak p}
\newcommand{\qq}{\mathfrak q}
\newcommand{\uu}{\mathfrak u}

\newcommand{\g}[1]{\mathfrak{#1}}
\newcommand{\ovl}{\overline}
\renewcommand{\to}{\longrightarrow}
\newcommand{\bm}{\boldsymbol}

\newtheorem{thm}{Theorem}[section]
\newtheorem{lem}[thm]{Lemma}
\newtheorem{cor}[thm]{Corollary}
\newtheorem{prop}[thm]{Proposition}

\theoremstyle{definition}
\newtheorem{dfn}[thm]{Definition}
\newtheorem{question}[thm]{Question}

\newtheorem{setting}[thm]{Setting}
\newtheorem*{claim*}{Claim}
\newtheorem{ex}[thm]{Example}
\newtheorem{rmk}[thm]{Remark}

\title[Fiber  criteria for flatness]{Fiber   criteria for flatness and homomorphisms of \\flat affine group schemes} 
\author[P.H. Hai]{Ph\`ung H\^o Hai}
\address[P.H.Hai]{Institute of Mathematics, Vietnam Academy of Science and Technology, 18 Hoang Quoc Viet, Hanoi, Vietnam}
\email{phung@math.ac.vn}
\author[H.D. Nguyen]{Hop D. Nguyen}
\address[H.D. Nguyen]{Institute of Mathematics, Vietnam Academy of Science and Technology, 18 Hoang Quoc Viet, Hanoi, Vietnam}
\email{ngdhop@gmail.com}
\author[J.P. dos Santos]{Jo\~ao Pedro   dos Santos}
\address[J.P.  dos Santos]{Institut Montpelli\'erain Alexander Grothendieck, Universit\'e de Montpellier,
Place Eug\`ene Bataillon, 34090 Montpellier, France}
\email{joao\_pedro.dos\_santos@yahoo.com}

\begin{document}
\date{\today} 

\begin{abstract}A very useful result concerning flatness  in Algebraic Geometry is EGA's ``fiber'' criterion.  
We propose similar fiber criteria to verify flatness of a module while  avoiding  ``finiteness'' assumptions. Motivated by a Tannakian  viewpoint (where the category of representations comes to the front),  we derive applications to the theory of affine and flat group schemes. 
\end{abstract}

\makeatletter
\@namedef{subjclassname@2020}{\textup{2020} Mathematics Subject Classification}
\makeatother

\maketitle

\section{Introduction}
A very useful result concerning flatness  in Algebraic Geometry is EGA's fiber    criterion for flatness \cite[$\mathrm{IV}_3$, 11.3.10]{ega}. (See also Lemma \ref{lem.08.09.21--1} for a simple statement.) This important result requires certain ``finiteness'' assumptions on the given objects, and it is the objective of the present paper to explore what can be said when the ``finiteness'' conditions are dropped. The interest behind this comes mainly  from the first and third named authors' works on Tannakian group schemes:  these   quite easily fail to be of finite type. But  the reader will have no difficulty in finding other instances where a more general result would be useful: classifying spaces, universal coverings, etc. 

Here is our first main result, which is also our first goal. 
\begin{thm}[= Theorem \ref{prop_fiber}]\label{main1_intro}
Let $R$ be a noetherian ring and   $f:R\to A $ be a flat  homomorphisms of rings. Let $M$ be an $A$-module.  Assume that   $M$ is flat over $R$ and 
\[\textup{(Fiber flatness)} \quad \text{for all $\mathfrak p\in {\rm Spec}\,  R$, the $A\oti_R \bm k({\mathfrak p})$-module $M\oti_R\bm k({\mathfrak p})$ is flat.}\] Then $M$ is flat over $A$.
\end{thm}

It should be noticed   that in contrast to EGA's fiber criterion (see Lemma \ref{lem.08.09.21--1}) --- which requires flatness at one single fiber --- here we make use of flatness for {\it all}
fibers. Moreover, in Theorem \ref{main1_intro}, we also require flatness of   the map $f:R\to A$ (unnecessary for the statement in \cite{ega}); this time, it is possible to replace this assumption by something finer, leading us to our second main result.

\begin{thm}[= Corollary \ref{cor_fibercriterion_prime}]\label{main2_intro}
Let $R$ be a noetherian ring and  $f:R\to A$ a ring map. Let $M$ be an $A$-module. For any $\pp\in \Spec\,R$, let $t(A/\pp A)$ be the torsion submodule of the $(R/\pp)$-module $A/\pp A$. Then $M$ is a flat $A$-module if and only if for every  $\pp \in \Spec\,R$, the following conditions hold:
\begin{enumerate}[\quad \rm (i)]
 \item \textup{(Tor vanishing over primes)}  $\Tor^A_1(A/\pp A,M)=$ $\Tor^A_1\left(\dfrac{A/\pp A}{t(A/\pp A)},M\right)=0$;
 \item \textup{(Fiber flatness)} The $\left(A\oti_R {\boldsymbol k}(\pp)\right)$-module $M\oti_R {\boldsymbol k}(\pp)$ is  flat.
\end{enumerate}
\end{thm}

It is important to observe that the hypothesis in Theorem \ref{main2_intro}
are weaker than those in Theorem \ref{main1_intro}, as shall be explained in the body of the paper. But we thought useful to deal with Theorem \ref{main1_intro} separately  in order to highlight  the simplicity of its statement and proof.

As mentioned at the beginning of this introduction, one reason for considering this sort of Commutative Algebra is motivated by its employment in the theory of {\it affine group schemes}. When dealing with these objects, it is a relevant question to ask how much a certain property of a morphism is determined by the functor  of  ``restriction of representations''.   In this direction, we can put forward the following.  

\begin{thm}[= Theorem \ref{surjective_criterion_group_hom}]
Let $R$ be a noetherian ring and $f:G'\to G$ be a homomorphism of flat affine group schemes over $R$. Denote by $f^*: \mathbf{Rep}_{\mathbf f} \,G  \to \mathbf{Rep}_{\mathbf f} \,G'$ the induced functor on categories of representations on finite $R$-modules. Then $f$ is  faithfully flat if and only if  $f^*$ is fully faithful and its essential image is stable under taking  subobjects.
\end{thm}

We now  summarize the contents of the remaining sections in the paper. In Section \ref{preliminar},  we recall some basic facts concerning the original fiber-by-fiber flatness criterion (cf. Section \ref{preliminar_classical}) and the rudiments of the theory of affine group schemes (cf. Section \ref{prelimina_group_schemes}). These summaries shall render the text more self-contained. They may also be of use to readers having a background in  Commutative Algebra, but less experience in the these more specific themes. Then,   in Section \ref{sect_criteria_flat}, we concentrate on a direct proof of Theorem \ref{main1_intro} by making decisive use of the $R$-flatness assumption on the $A$-module $M$ in order to compute Tor modules. That  Theorem \ref{main1_intro} is less general than   Theorem \ref{main2_intro} is explained by Remark \ref{rmk_torsionfree_hypothesis}.  
Section \ref{sect_criteria_Tor} is dedicated to  Theorem \ref{main2_intro}. The key result is the slightly less general Theorem \ref{thm_fibercriterion_strong}.
   Section \ref{sect_purity} presents  criteria for purity of maps between modules employing the material already developed in the previous parts and with an eye towards applications in Section \ref{sect_applications}. 
Section \ref{sect_counterexamples} provides some counter-examples to show how sharp our hypothesis are. We then conclude the work  with   Section \ref{sect_applications}, which  discusses   morphisms of flat affine group schemes over an arbitrary  noetherian  base under the light of our fibre criteria for flatness.

We end this introduction by noting that Theorem \ref{main1_intro} was also recently discovered  in the work \cite{barthel-et-al23}; there, it is claimed that the method of proof is already visible in Expos\'e XVII in SGA 4.


\subsection*{Notations and conventions}
\begin{enumerate}[(1)]
\item All rings are commutative and unital, and all morphisms of rings are unital. 

\item If $A$ is a ring and $\g p$ is a prime ideal of $A$, we let $\bm k(\g p)$ stand for the quotient field of the domain  $A/\g p$.

\item The category of groups is denoted by $\Grp$. That of $R$-algebras is denoted by $\algebras R$.
  
\item Given a ring  $A$ and an    $A$-module  $M$, we let $\mathrm{ann}_A(M)$ stand for the ideal $\{a\in A\,:\,\text{$am=0$ for all $m\in M$}\}$. If $m\in M$, we write $\mathrm{ann}_A(m)=\{a\in A\,:\,\text{$am=0$}\}$. 
\item Given a ring  $A$,  an    $A$-module  $M$,  an ideal $\g a\subset A$ and a submodule $N\subset M$, we let $(N:\g a)_M$ be the $A$-submodule $\{m\in M\,:\,\g a\cdot m\in N\}$. If $\g a=aA$,   we then write $(N: a)_M$ instead of $(N:\g a)_M$. 
\end{enumerate}

\section{Preliminaries}\label{preliminar}

\subsection{Grothendieck's flatness criteria}\label{preliminar_classical}
\begin{setting}
\label{setting_10.08.21--1}
 Let $g:X\to S$, $h:Y\to S$, be morphisms of schemes and $f:X\to Y$ be an $S$-morphism (i.e. $h\circ f=g$). Let $x$ be a point of $X$, $y=f(x)$, $s=g(x)$ and let $\mathcal F$  be a quasi-coherent sheaf of $\mathcal O_X$-modules. 
\[
\xymatrix{
x\in X\ar[rr]^f\ar[rd]_g&&Y\ni y\ar[ld]\\
&s\in S.}
\]
The fiber of $g$, $h$  at the point $s\in S$ are the base change of these morphisms with respect to the morphism $s:\Spec\,\boldsymbol k(s)\to S,$ where $\boldsymbol  k(s)$ is the residue field of $s$. 
\[\xymatrix{\ar@{}[rd]|{\Box}
X_s\ar[r]\ar[d]_{f_s}&X\ar[d]^f\\
\Spec\,\boldsymbol  k(s)\ar[r]_-{s}&X.}\]
\end{setting}

We say that $\mathcal F$ is $f$-flat at $X$ if $\mathcal F_x$ -- the stalk of $\mathcal F$ at $x$ -- is a flat module over the local ring $\mathcal O_{Y,y}$ by means of $f$. 

\begin{thm}[Grothendieck's fiber criterion for flatness] \cite[$\mathrm{IV}_3$, 11.3.10]{ega} \label{ega_flatness}
Let us adopt the notations and conventions of Setting \ref{setting_10.08.21--1}. Assume that $\mathcal F_x\neq 0$
and one of the following conditions are satisfied:
\begin{enumerate}[\quad\rm(i)]
\item The schemes $S, X$ and $Y$ are locally noetherian and $\mathcal F$ is coherent;
\item The morphisms $g$ and $h$ are locally of finite presentation and $\mathcal F$ is of finite presentation.
\end{enumerate}  
Then the following are equivalent:
\begin{enumerate}
[\quad\rm(1)]\item The $\mathcal O_X$-module
$\mathcal F$ is $g$-flat at $x$ and $\mathcal F|_{X_s}$, its  pull-back to $X_s$,   is $f_s$-flat at $x$;
\item  The $\mathcal O_X$-module $\mathcal F$ is $f$-flat at $x$ and $h$ is flat at $y$.  
\end{enumerate}
\end{thm}



One remarkable feature of Theorem \ref{ega_flatness} is that  flatness of the morphism \emph{at the closed fiber} implies   flatness in a   neighborhood. This property holds only under a certain finiteness condition.

The proof of the claim under assumption (ii) is reduced to the assumption that $S$ is noetherian (\cite[$\mathrm{IV}_3$, 11.2.7]{ega}) and thus is a sub-case of assumption (i).  On the other hand, if $S, X$ and $Y$ are noetherian schemes, then the assumption of finite presentation on $g$ and $h$ can be dropped.
 The claim under assumption (i) amounts to the following lemma.
\begin{lem}\cite[$\mathrm{IV}_3$, 11.3.10.1]{ega}
\label{lem.08.09.21--1}
Let $\rho:R\to A$, $\sigma:A\to B$ be local homomorphisms of local noetherian rings and $M$ be a nonzero $B$-module of finite type. Let $k$ be the residue field of $R$. The following are equivalent:
\begin{enumerate}
[\quad\rm(1)]
\item The $R$-module $M$ is flat and the $A\oti_Rk$-module $M\oti_Rk$ is   flat;
\item The $A$-module $M$ is  flat and $A$ is $R$-flat.  
\end{enumerate} 
\end{lem} 

This lemma, in turn, is derived from the following famous lemma, which shall also be used in the course of this work.

 \begin{lem}[Local criterion for flatness, {\cite[$0_{\mathrm{III}}$,~10.2]{ega},    \cite[Thm. 22.3]{Mat89}}]
 \label{local_flat}
  
Let $A$ be a ring with an ideal $\mathfrak a$ and  $M$ be an (arbitrary) module.  Assume that one of the following condition is satisfied:
\begin{enumerate}
[\quad\rm (i)] 
\item $\mathfrak a$ is nilpotent;
\item There exists an $A$-algebra $B$ which is noetherian, such that $\mathfrak aB$ lies in all maximal ideals of $B$ and that $M$ admits a structure of $B$-module of finite type compatible with that of $A$. 
\end{enumerate} 
 Then the following are equivalent:
\begin{enumerate}
[\quad\rm(1)]
\item  The $A$-module $M$ is flat;
\item The $A/\mathfrak a$-module $M/\mathfrak aM$ is flat and the natural (surjective) map
$$\gamma^{\mathfrak a}_{M/A}:{\rm gr}^{\mathfrak a}(A)\oti_{{\rm gr}^{\mathfrak a}_0(A)} {\rm gr}^{\mathfrak a}_0(M)\longrightarrow {\rm gr}^{\mathfrak a}(M)$$
is an isomorphism. Here, as usual,  
\[
{\rm gr}^\afr(M)=\bigoplus_{n=0}^\infty \frac{\afr ^n M}{\afr^{n+1} M}
\] denotes associated graded module of $M$ with respect to $\afr$, and 
\[
{\rm gr}^\afr_n(M)= \frac{\afr ^n M}{\afr^{n+1} M}
\] 
its   homogeneous component of degree $n$.
\item The $A/\mathfrak a$ module $M/\mathfrak aM$ is flat and the multiplication $\mathfrak a\oti_A M\to M$ is injective.
\item For all $n\geq 1$, the $A/\mathfrak a^n$-module $M/\mathfrak a^nM$ is flat.
 \end{enumerate}
\end{lem}
We notice that the equivalence of conditions $(3)$ and $(4)$ holds without any assumption on $M$ and $\mathfrak a$.
Therefore, we have the following lemma, which will be used later on.
\begin{lem}
\label{lem_localcriterion_generalform}
Let $A$ be a ring, ${\mathfrak a}$ an ideal and $M$ an $A$-module.  Assume that $M/{\mathfrak a}M$ is a flat $(A/{\mathfrak a})$-module, and that $\Tor^A_1(A/{\mathfrak a},M)=0$. Then for any $n\ge 2$, $M/{\mathfrak a}^nM$ is a flat $(A/{\mathfrak a}^n)$-module, and $\Tor^A_1(M,N)=0$ any $(A/{\mathfrak a}^n)$-module $N$.
\end{lem}
\begin{proof} 
The first claim is the implication $(3)\Longrightarrow(4)$ in Lemma \ref{local_flat}. We show the second claim. For $n=1$, consider an exact sequence $0\to P\to F\to N\to 0$, where $F$ is a free $A/{\mathfrak a}$-module. 
Tensoring with $M$ and noticing that $\Tor^A_1(A/{\mathfrak a},M)=0$ we obtain an exact sequence
$$0\to \Tor^A_1(N,M)\to P\oti_AM\to F\oti_AM.$$
As $M/{\mathfrak a}M=M\oti_AA/{\mathfrak a}$ is $A/{\mathfrak a}$-flat, the map $P\oti_AM\to F\oti_AM$ is injective. Hence $ \Tor^A_1(N,M)=0$.
If we know that $\Tor^A_1(A/{\mathfrak a}^n,M)=0$ for any $n\geq 2$ then the discussion above will yield  $\Tor^A_1(M,N)=0$ any $(A/{\mathfrak a}^n)$-module $N$.
 But to show $\Tor^A_1(A/{\mathfrak a}^n,M)=0$, we only need to use induction on $n\ge 2$ and the exact sequence \[0=\Tor^A_1({\mathfrak a}^{n-1}/\g a^n,M) \to \Tor^A_1(A/{\mathfrak a}^n,M) \to \Tor^A_1(A/{\mathfrak a}^{n-1},M)\] (note that ${\mathfrak a}^{n-1}/{\mathfrak a}^n$ is annihilated by ${\mathfrak a}$).
\end{proof}

\subsection{Affine group schemes and Tannakian duality}\label{prelimina_group_schemes}

A much valuable  account of the theory of  affine group schemes can be found in \cite[Chapters 1 and 2]{Wat} and in \cite[Part I, 2.1--2.4]{jantzen}; the theory of representations of affine group schemes is masterfully presented as well in op.cit., see specially 
\cite[Ch. 3]{Wat} and \cite[Part I, 2.7--9]{jantzen}. With these references in mind, we make a brief presentation for the convenience of the reader. 

   Given a ring $R$ and an affine and flat  group scheme $G$ over $R$, we let $\mathbf{Rep}_{\mathbf f}\,G$ stand for the category of representations of $G$ on finite $R$-modules. (For the notion of representation, see \cite[Ch. 3]{Wat} and \cite[Part I, 2.7--9]{jantzen}.)

Let $R$ be a base ring. An \emph{affine group scheme} over $R$ is a pair $(G,m)$ with $G$ an affine scheme over $R$ and $m: G\times G \to G$ (the \emph{multiplication map}) a morphism of schemes over $R$ such that for all $R$-algebra $S$, the induced map $m(S): G(S)\times G(S) \to G(S)$ is a group structure on the set $G(S)$. (For the sake of readability, mention to $m$ will frequently be omitted.) Equivalently, an affine group scheme over $R$ is a functor $G: \algebras R \to \Grp$ whose  underlying set-valued functor is representable. 

A \emph{morphism} $(G',m') \to (G,m)$ of affine group schemes over $R$ is a morphism $f: G'\to G$ of $R$-schemes  which is compatible with the multiplication in the sense that for each $R$-algebra $S$, the map $G(S)\to G'(S)$ is a morphism of groups. Alternatively, we require the  following diagram to commute:
\[
\xymatrix{
G'\times G' \ar[r]^{f\times f} \ar[d]_{m'} & G \times G \ar[d]^{m} \\
G' \ar[r]_f & G. 
}
\]
Following tradition, given an affine group scheme over $R$, we let $R[G]$ stand for its ring of functions. The existence of multiplication $m:G\times G\to G$, a neutral element $e:\mathrm{Spec}\,R\to G$ and an ``inversion morphism'' $G\to G$ all reflect on properties of $R[G]$, which lead to the axioms of commutative  Hopf algebras, cf.  \cite[1.4]{Wat} and \cite[Part I, 2.3]{jantzen}. (We warn the reader that what \cite{Wat} calls a Hopf algebra is usually just called a {\it commutative} Hopf algebra. The reference \cite{jantzen} does make this distinction, see p. 21 in op.cit.)

\begin{ex}
The \emph{additive group} $\mathbf G_a$ (over $R$), whose underlying scheme is $\mathbf  A^1_R$, is the affine group scheme represented by the $R$-algebra $R[x]$. Thus for any $R$-algebra $S$, 
\[\begin{split}
\mathbf  G_a(S)&=\Hom_{\algebras R}(R[x],S)\\&=(S,+)\\&=\text{the additive group of $S$.}
\end{split}
\]
 The multiplication map $m: \mathbf  G_a \times \mathbf  G_a =\Spec(R[x]\oti_R R[x]) \to \mathbf  G_a=\Spec R[x]$ is induced by the map $R[x] \to R[x]\oti_R R[x], x\mapsto x\oti 1+ 1\oti x$. 
\end{ex}

\begin{ex}A very simple way of producing group schemes which are not necessarily of finite type is via ``multiplicative'' group schemes. Let $\Lambda$ be an abelian group and let $R[\Lambda]$ be the associated group algebra. Define $\mathbf{Diag}(\Lambda)$ as the affine scheme $\mathrm{Spec}\,R[\Lambda]$ so that $\mathbf{Diag}(\Lambda)(S)=\mathrm{Hom}_{\Grp}(\Lambda,S^\times)$. There is an obvious group structure on $\mathrm{Hom}_{\Grp}(\Lambda,S^\times)$ so that the multiplication morphism $\mathbf{Diag}(\Lambda)\times\mathbf{Diag}(\Lambda)\to\mathbf{Diag}(\Lambda)$ is given, on the level of algebras, by  $\sum r_\lambda\cdot\lambda\longmapsto\sum r_\lambda\cdot\lambda\otimes\lambda$.
\end{ex}

The theory of affine group schemes is an extension of the theory of linear algebraic groups and as such, a large part of it is dedicated to the study of ``linear'' representations, which we now set out to explain briefly. 

 Given an $R$-module $V$, let $\mathbf{GL}_V$ stand for the functor $\algebras R\to \Grp$ which associates to $S$ the group of $S$-linear endomorphisms of $S\otimes_RV$. In certain cases $\mathbf{GL}_V$ is representable by an affine group scheme, but this is not unconditionally  true. (The verification of the positive case can be found in  \cite[Part I, 2.2]{jantzen}, while counter-examples can be found in \cite{nitsure}.)
A linear representation of an affine group scheme  $G$ over $R$ is the data of an $R$-module $V$ plus a natural transformation of functors $G\to \mathbf{GL}_V$. Little effort is required to construct from these {\it the category of representations $\mathbf{Rep}\,G$}. (In \cite{jantzen} these are called $G$-modules.)

Now, the apparent amount of information   encoded in a natural transformation of functions   can be off-putting. For this reason, a more wieldy object is that of a {\it co-module}. 
These   are clearly explained in \cite{Wat} and \cite{jantzen}, and we content to give  peripheral explanations. 

So let $V$ be an $R$-module and $\rho:G\to \mathbf{GL}_V$ be a representation. We then have an isomorphism of $R[G]$-modules 
\[
V\otimes_RR[G] \to  V\otimes_RR[G]
\]
obtained from the image in 
$\mathbf{GL}_V(R[G])$ of $\mathrm{id}\in G(R[G])$. This, in turn, gives an $R$-linear map 
\[
\rho_V:V\to V\otimes_RR[G]. 
\]
The fact that $\rho(S)$ is a morphism of groups for each $S\in \algebras R$ then allows us to show that $\rho_V$ has to respect other axioms  (cf. \cite[3.2, Theorem]{Wat} or \cite[Part I, 2.8]{jantzen})  which then define, what is called, a {\it  co-action  of the Hopf algebra $R[G]$} on the $R$-module $V$. (In fact, the {\it algebra} structure of $R[G]$ plays no role and the axioms for a co-action only require the {\it co-algebra} structure of $R[G]$.) 
The data of an $R$-module plus a co-action defines a {\it co-module} over $R[G]$. It takes little effort to define what morphisms of co-modules are and we then obtain a category $\Comod {R[G]}$  of $R[G]$-comodules.

Summarizing, by starting with a representation $G\to\mathbf{GL}_V$, we obtain a structure of $R[G]$-comodule on $V$. It turns out that this association is functorial and we arrive at an equivalence  \[\mathbf{Rep}\,G\stackrel\sim\to\Comod{R[G]}.\]

When $G$ is a {\it flat} $R$-scheme, it is possible to endow the kernel and cokernel of an arrow of $\Comod {R[G]}$   with the structure of a co-module (cf. \cite[1.3]{Se68} or \cite[2.9]{jantzen}) in such a way that the category of representations becomes an abelian category $\mathbf{Rep}\,G$. If $R$ is in addition noetherian, then an even more pertinent category is $\mathbf{Rep}_{\mathbf f} \,G$, the full subcategory of $\mathbf{Rep} \,G$ whose objects underlie finite $R$-modules. 

Let us end this section by explaining the {\it categorical Tannakian philosophy}. In a nutshell, this philosophy says that the study of affine flat group schemes and morphisms between them is to be approached via the categories $\mathbf{Rep}$ or $\mathbf{Rep}_{\mathbf f}$. 
This philosophy has great success when $R$ is a field \cite{DM82} or a Dedekind domain \cite{DH18}.  

One of the questions addressed by the categorical Tannakian philosophy is the derivation of properties of a morphism of affine flat group schemes from the associated functor on representation categories. More precisely, let $f:G'\to G$ be a morphism of flat and affine group schemes and denote by \[f^*:\mathbf{Rep}_{\mathbf f}\,G\to\mathbf{Rep}_{\mathbf f}\,G'\]
the ``restriction'' functor \cite[Part I, 6.17]{jantzen}. What properties of $f$ can be derived from properties of $f^*$? We shall apply our results on pure morphisms  to study this question, see Section \ref{12.12.2023--1}.

\section{Fiber criteria for flatness}
\label{sect_criteria_flat}

The aim of this work is to relieve the finiteness assumption (noetherian or finite presentation assumption) in the above claims. Of course, one cannot have it for free. Our main assumption will be the flatness of morphisms at {\em all} fibers.

\begin{thm}[Fiber criterion of flatness]
\label{thm_fiber_geom}
Let us adopt the notations of  Setting \ref{setting_10.08.21--1}. 
Assume that
\begin{enumerate}[\quad\rm(i)]
\item the morphisms $g$ and $h$ are flat;
\item the base-change morphism $f_s$ is flat (resp. faithfully flat) for all points $s$ of $S$; and, 
\item the scheme $S$ is locally noetherian.
\end{enumerate}
 Then $f$ is flat (resp. faithfully flat).
\end{thm}  

The claim of theorem is local, so that we can assume that $X$,  $Y$ and $S$ are affine schemes. Hence, Theorem \ref{thm_fiber_geom} follows from Theorem \ref{prop_fiber} and Corollary \ref{cor_fiber_faithfullyflat}, which are in fact stronger claims.

\begin{thm} 
\label{prop_fiber} Let $R$ be a noetherian ring and   $f:R\to A $ be a flat  homomorphisms of rings. Let $M$ be an $A$-module.  Assume that   $M$ is flat over $R$ and 
\[\tag{FF}\text{for all $\mathfrak p\in {\rm Spec}\,  R$, the $A\oti_R \bm k({\mathfrak p})$-module $M\oti_R\bm k({\mathfrak p})$ is flat.}\] Then $M$ is flat over $A$.
\end{thm}

\begin{proof}
We begin by noting that the condition (FF) is ``stable under quotients of $R$'' in the following sense: if $\g a\subset R$ is an ideal,  then (FF) holds   once we replace $R$ by $R/\g a$, $A$ by $A/\g aA$ and $M$ by $M/\g aM$. Obviously, flatness of $A/\g aA$ and of $M/\g aM$ over $R/\g a$ keep on holding as well. 

In order to prove the theorem, we need to   show that
\[
\tag{$*$}{\rm Tor}_i^A(M,N)=0, \text{\rm  for all } i\geq 1, 
\]
for an arbitrary $A$-module $N$.  
We notice that for an exact sequence $0\to N_1\to N_2\to N_3\to 0$, if $(*)$ holds for $N_1$ and $N_3$,   then it also holds for $N_2$; similarly, if $(*)$ holds for    $N_2$ and $N_3$, then it holds for   $N_1$.    
The proof or $(*)$ will be done in several steps.

\bigskip
\noindent {\em Step 0}. 
Here we explain the mechanism behind much that is to follow. Let us   give ourselves a resolution of $M$ by flat $A$-modules   
 $P_\bullet\longrightarrow M$, so that 
\[
{\rm Tor}_i^A(M,N)\simeq H_i(P_\bullet\oti_AN).
\]
Since $A$ and $M$ are $R$-flat  then, for any ideal $\g a\subset R$,   the complex of $R/\g a$-modules $P_\bullet/\g aP_\bullet\to M/\g aM$ is also a   resolution of $M/\g aM$ by flat $R/\g a$-modules.  Indeed, note that  $H_{i}(P_\bullet/\g aP_\bullet )\simeq {\rm Tor}_i^R(M,R/\g a)=0$ for $i>0$.  
Now, if $\g a\subset\mathrm{ann}_R(N)$, then the canonical vertical arrows in 
\[\xymatrix{
\cdots\ar[r]&P_1\oti_AN\ar[r]\ar[d]^\sim&P_0\oti_AN\ar[r]\ar[d]^\sim&M\oti_AN\ar[r]\ar[d]^\sim&0
\\
\cdots\ar[r]&P_1/\g aP_1\underset {A/\g aA}{\oti}N\ar[r]&P_0/\g aP_0\underset {A/\g aA}{\oti} N\ar[r]&M/\g aM\underset {A/\g aA}{\oti}N\ar[r]&0
}
\]
are isomorphisms  \cite[II.3.6, Corollary 3]{bourbaki_algebra} so that
\[\tag{$\dagger$}
\Tor_i^A(M,N)\simeq\Tor_i^{A/\g aA}(M/\g aM,N).\]
This will allow us to work inductively.

\bigskip
\noindent {\em Step 1}. We assume that $R$ is integral and proceed by induction on $\dim R$; the case $\dim R=0$  is obvious. 
We therefore suppose that  $\dim R>0$ and that $(*)$ is true for all domains of dimension strictly smaller than $\dim R$.

\noindent \textbf{Case 1a}: $\mathrm{ann}_R(N)$ contains a product of powers of distinct non-zero primes $\g p^{e_1}_1\cdots\g p_n^{e_n}$.

We proceed by induction on $\sum e_j$ to prove $(*)$. We have an exact sequence of $A$-modules 
\[0\to (0:\g p_1)_N\to N\to \frac{N}{(0:\g p_1)_N}\to0\]
where 
\[
{\rm ann}_R\,(0:\g p_1)_N\subset\g p_1\quad\text{and}\quad{\rm  ann}_R\,\frac{N}{(0:\g p_1)_N}\subset\g p_1^{e_1-1}\g p_2^{e_2}\cdots\g p_n^{e_n}.
\]
It is then enough to verify the case where $N$ is annihilated by a prime $\g p\in \mathrm{Spec}\,R$. 
Using $(\dagger)$, we conclude that
$\Tor_i^A(M,N)\simeq \Tor_i^{A/\g pA}(M/\g pM,N)$. 
Now, it is not difficult to see that the hypothesis of the Theorem are verified for  $R/\g p$, $A/\g pA$ and $M/\g pM$ and since $\dim R/\g p<\dim R$, we may conclude by induction that $\Tor_i^A(M,N)=0$ for $i>0$. 

\noindent \textbf{Case 1b}: $\mathrm{ann}_R(N)\not=0$. Each non-zero ideal $\g a\subset R$ contains a product of powers of non-zero primes, by primary decomposition. Thus using Case 1a, $\mathrm{Tor}_i^A(M,N)=0$ for all $i>0$.

\noindent \textbf{Case 1c}: $N$ is an $R$-torsion module.

In this case, $N=\varinjlim_{\lambda}N_\lambda$
where each $N_\lambda$ is an $A$-module such that $\mathrm{ann}_R(N_\lambda)\not=0$.  We then apply Case 1b,    and the isomorphism  
\[\mathrm{Tor}_i^A(M,N)\simeq \varinjlim_{\lambda}\mathrm{Tor}_i^A(M, N_\lambda).\]

\noindent \textbf{Case 1d}: $N$ is a torsionfree  $R$-module.

Let $K$ be the quotient field of $R$. Then $N$ is an $A$-submodule of $N_K:=N\oti_RK$  and the quotient   $N_K/N$, which is an  $A$-module,  is an $R$-torsion module. Case 1c implies that $\mathrm{Tor}_i^A(M,N_K/N)=0$ for all $i>0$. 

On the other hand, by assumption $M_K:=M\oti_RK$ is flat over $A_K:=A\oti_RK$, hence
${\rm Tor}_i^A(M,N_K)={\rm Tor}_i^{A_K}(M_K,N_K)=0$, for all $i>0$. 
Thus we get ${\rm Tor}_i^A(M,N)=0$ for all $i>0$ by means of the exact sequence $0\to N\to N_K\to N_K/N\to0$.

\noindent \textbf{Case 1e}: $N$ is arbitrary. Let 
\[T=\{n\in N\,:\,\text{$rn=0$ for some $r\in R\setminus\{0\}$}\}.
\]This is an $A$-submodule of $N$ such that  the  $R$-module $N/T$ is torsionfree. Applying Case 1d to $N/T$, Case 1c to $T$, 
and the long exact sequence associated to $0\to T\to N\to N/T\to0$ shows that $(*)$ is true for $N$. 
\bigskip 

\noindent{\em Step 2.} $R$ is not a domain. Taking an arbitrary $A$-module $N$, we wish to prove $(*)$. Using primary decomposition, we find (not necessary pairwise distinct) non-zero prime ideals $\pp_1,\ldots,\pp_n$ of $R$ whose product is zero. For each prime ideal $\pp \in \Spec\,R$, using Step 1 for the flat map $R/\pp \to A/\pp A$ and the module $M/\pp M$, we see that $M/\pp M$ is flat over $A/\pp A$. Fix an integer $i>0$. Consider the filtration:
\[
\underbrace{N}_{F^0}\supset \underbrace{\mathfrak p_1N}_{F^1}\supset \underbrace{\mathfrak p_2\mathfrak p_1N}_{F^2}\supset\ldots\supset \underbrace{\mathfrak p_n \ldots\mathfrak p_2\mathfrak p_{1}N}_{F^n}= 0,
\]
so that, writing $G^\nu=F^{\nu-1}/F^{\nu}$, we have $\mathrm{ann}_R\,G_\nu\subset\g p_\nu$. Using $(\dagger)$, we conclude that $\Tor_i^A(M,G_\nu)\simeq\Tor_i^{A/\g p_\nu A}(M/\g p_\nu M,G_\nu)$. Because of Step 1, applied to $R/\g p_\nu$, $A/\g p_\nu A$ and $M/\g p_\nu M$, we conclude that $\Tor_i^A(M,G_\nu)=0$. A simple argument with exact sequences then shows that $\Tor_i^A(M,N)=0$.  
\end{proof}


\begin{cor}
  \label{cor_fiber_faithfullyflat} 
  Let $f:R\to A$ be a flat homomorphisms rings,  with  $R$  Noetherian. Let $M$ be an $A$ module. Assume that
$M$ is flat over $R$ and for all $\g p\in {\rm Spec} \, R$, $M\oti_R\bm k({\mathfrak p})$ is faithfully flat over $A\oti_R\bm k({\mathfrak p})$. Then $M$ is faithfully flat over $A$.
\end{cor}
\begin{proof}  Theorem
 \ref{prop_fiber} implies that $M$ is flat over $A$. We show the faithful flatness. 
According to \cite[7.2, p. 47]{Mat89}, it suffices to show that 
$M\oti_A\bm k(\g m) \neq 0$ for each maximal ideal $\mathfrak m$ of $A$.
Let $\g p\in\mathrm{Spec}\,R$ lie below $\g m$. Note that $A\to \bm k(\g m)$ factors as $A\to A\oti_R\bm k(\g p)\to \bm k(\g m)$.
Therefore 
\[
M\oti_A\bm k(\g m) \simeq  (M\oti_{A}(A\oti_R\bm k(\g p)))\underset {A\oti_R\bm k(\g p)}{\oti} \bm k(\g m).
\]
By assumption, $M\oti_{A}(A\oti_R\bm k(\g p))$ is faithfully flat over $A\oti_R\bm k(\g p)$, hence the right-hand-side above is non-zero.
\end{proof}

\section{Further fiber criteria for flatness}
\label{sect_criteria_Tor}

Let $R$ be a noetherian ring and  $f:R\to A$ be a morphism of rings. Let $M$ be an $A$-module. The assumptions  in Theorem
 \ref{prop_fiber} may be considerably weakened: instead of flatness of  $M$ and $A$ over $R$, we shall require simply  the that    $\text{Tor}_1^A(M,-)$ vanishes for a certain family of $A$-modules (specified in  Theorem \ref{thm_fibercriterion_prime}).

We first fix some notations.
%
%
Given an integral domain $D$ and a $D$-module  $N$, let $t(N)$ denote the {\it torsion submodule} of $N$.
The quotient 
$N_{\tf}=N/t(N)$ 
is the  {\it torsionfree} part of $N$. We note that $t(t(N))=t(N)$, and $t(N_{\tf})=0$ (i.e. $t(N)$ is a torsion module, while $N_{\tf}$ has no torsion).
These constructions shall be applied in the following situation. 
Given a ring homomorphism $R\to A$ and a prime ideal $\pp \in \Spec\,R$, we obtain an $R/\g p$-module $A/\g pA$, to which we apply the above constructions and arrive at 
\[
t(A/\g pA)\quad \text{and}\quad (A/\g pA)_{\tf}. 
\]
Thus $(A/\pp A)_{\tf}$ depends on the map $R\to A$ and the prime ideal $\pp$. 

\begin{thm}[Fiber criterion for flatness]
\label{thm_fibercriterion_prime}
Let $R$ be a noetherian ring and 
$f:R\to A$ a homomorphism. Let $M$ be an $A$-module. Assume that for every  $\pp \in \Spec\,R$, the following conditions hold:
\begin{enumerate}[\quad \rm (i)]
 \item the canonical  map $\g pA\oti_AM\to\g pM$ is injective;  alternatively, 
 $\Tor^A_1(A/\pp A,M)$ vanishes;
 \item The $\left(A\oti_R {\boldsymbol k}(\pp)\right)$-module $M\oti_R {\boldsymbol k}(\pp)$ is  flat;
  \item Either $A$ is noetherian or $\Tor^A_1((A/\pp A)_{\tf},M)=0$.
\end{enumerate}
Then $M$ is a flat $A$-module.
\end{thm}

\begin{rmk}[On the hypothesis in Theorem \ref{thm_fibercriterion_prime}]
\label{rmk_torsionfree_hypothesis}
Let us adopt the notations in the statement of Theorem \ref{thm_fibercriterion_prime}.
  Given an ideal  $\g l\subset R$, we consider the composition map  
\[
\xymatrix{(\g l\oti_RA)\oti_A M\ar[rr]^-{\alpha\oti{\rm id}_M}&&\g lA\oti_AM\ar[r]^-{\beta}&\g lM};
\]
clearly $\beta$ and $\alpha\oti{\rm id}$ are surjective, and if $M$ is $R$-flat, then $\beta\circ(\alpha\oti{\rm id})$ is an isomorphism, so that $\beta$ is likewise. Therefore, if $M$ is $R$-flat, then $\Tor_1^A(A/\g lA,M)=0$ for any $\g l\subset  R$:  condition (i) of Theorem \ref{thm_fibercriterion_prime} holds. In addition, if $A$ is $R$-flat, then $A/\g pA\simeq (A/\g pA)_{\tf}$, which implies condition (iii) of the aforementioned theorem. Hence Theorem \ref{prop_fiber} follows from Theorem \ref{thm_fibercriterion_prime}. 

We end this remark by noting that  hypothesis (iii) of Theorem \ref{thm_fibercriterion_prime} may follow directly from (i), e.g.,   if $A$ is $R$-flat, or if $R$ is artinian.
\end{rmk}

We will deduce Theorem \ref{thm_fibercriterion_prime} from the following slightly weaker version, which we will treat first.

\begin{thm}
\label{thm_fibercriterion_strong}
Let $R\to A$ be a ring homomorphism, where $R$ is noetherian. Let $M$ be an $A$-module. Assume that the following conditions hold:
\begin{enumerate}[\quad \rm (F1)]
\item For every \textup{(}possibly non-prime\textup{)} ideal ${\mathfrak{l}}$ of $R$,  $\Tor^A_1(A/{\mathfrak l}A,M)=0$;
 \item The $\left(A\oti_R {\boldsymbol k}(\pp)\right)$-module $M\oti_R {\boldsymbol k}(\pp)$ is flat for every $\pp\in \Spec\,R$.
 \end{enumerate}
 Assume moreover that either of the following conditions are fulfilled:
 \begin{enumerate}[\quad \rm (F3a)]
 \item $A$ is a noetherian ring;
 \item $\Tor^A_1((A/\pp A)_{\tf},M)=0$ for every $\pp \in \Spec\,R$.
\end{enumerate}
Then $M$ is a flat $A$-module.
\end{thm}

For the proof of this theorem, we shall need some lemmas.

\begin{lem}[The torsion submodule localizes]
\label{lem_torsion_submodule_localizes}

Let $R\to A$ be a ring homomorphism, where $R$ is a domain. Let $\qq\in \Spec(A)$ be a prime ideal, $\pp=\qq\cap R$, and $\uu\subset \pp$ another prime ideal of $R$. Then there is an equality $t(A_\qq/\uu A_\qq)= t(A/\uu A)_\qq$ of submodules of $(A/\uu A)_\qq$, which is considered as an $(R/\uu)_\pp$-algebra.
\end{lem}
\begin{proof}
Replacing $R\to A$ by $R/\uu \to A/\uu A$, we can assume that $\uu=(0)$. We show that $t(A_\qq)=t(A)_\qq$, where $A_\qq$ is considered as an $R_\pp$-module. If $x\in t(A)$, then $ax=0$ for some $a\in R\smallsetminus \{0\}$. Now as $R$ is a domain, so is $R_\pp$ and for any $s\notin \qq$, $(a/1)(x/s)=0$, so $x/s\in t(A_\qq)$. Thus $t(A)_\qq \subset t(A_\qq)$.

Conversely, take $x/s\in t(A_\qq)$, where $x\in A, s\notin \qq$. Then $(a/s')(x/s)=0$ for some $a\in R\smallsetminus \{0\}, s'\in R\smallsetminus \pp$. This implies $wax=0$ for some $w\in  A\smallsetminus \qq$. Hence $wx\in t(A)$ and $x/s=wx/(ws)\in t(A)_\qq$, as desired.
\end{proof}

We have the following observation on the vanishing of certain Tor modules.

\begin{lem}
\label{lem_reduction_to_domains}
Let $A$ be noetherian and $M$ be an $A$-module. Let $z\in A$ be a (non-zero) element. Assume that $\Tor^A_1(M,A/zA)=0$ and $M/zM$ is a flat $(A/zA)$-module. Then 
$\Tor^A_i(M,N)=0$ for any $(A/zA)$-module $N$ and any $i\ge 1$. 
\end{lem}
\begin{proof}
According to Lemma \ref{lem_localcriterion_generalform},
the claim holds for $i=1$. 

Since $A$ is noetherian, the chain of ideals $(0:z)_A \subset (0: z^2)_A \subset \cdots$ stabilizes. Thus one can choose a finite $d\ge 1$ such that $(0: z^d)_A=(0: z^{d+1})_A$. 

Since $z^dA/z^{d+1}A$ is annihilated by $z$, $\Tor^A_1(M,z^dA/z^{d+1}A)=0$. Hence the exact sequence
\[
0\to z^dA \stackrel{\cdot z}\to z^dA \to z^dA/z^{d+1}A \to 0
\]
induces an exact sequence $\Tor^A_1(M,z^dA)\stackrel{\cdot z}\to \Tor^A_1(M,z^dA) \to 0$. Thus the multiplication by $z$ is a surjective map on $\Tor^A_1(M,z^dA)$. This implies that the homothety $\Tor^A_1(M,z^dA)\stackrel{\cdot z^d}\to \Tor^A_1(M,z^dA)$ is also surjective. On the other hand, the module $\Tor^A_1(M,z^dA) \simeq \Tor^A_2(M,A/z^dA)$ is annihilated by $z^d$. Hence $\Tor^A_1(M,z^dA) \simeq \Tor^A_2(M,A/z^dA)=0$.

Since $zN=0$, $N$ is also an $(A/z^dA)$-module. Let $0\to U\to F\to N \to 0$ be an exact sequence of $(A/z^d A)$-modules, where $F$ is free. Thanks to Lemma \ref{lem_localcriterion_generalform}, $M/z^dM$ is a flat $(A/z^dA)$-module and  $\Tor^A_1(M,A/z^dA)=\Tor^A_1(M,U)=0$. The fact that $F$ is free over $A/z^dA$, and $\Tor^A_2(M,A/z^dA)=0$ yield the exact sequence
\[
0= \Tor^A_2(M,F) \to \Tor^A_2(M,N) \to \Tor^A_1(M,U) =0.
\]
Hence $\Tor^A_2(M,N)=0$, as claimed.

Now we show that $\Tor^A_i(M,N)=0$ for every $i\ge 1$ and every $(A/zA)$-module $N$ by induction on $i$. This is known for $i=1,2$. Assume that $i\ge 3$. From the exact sequences
\[
\xymatrixrowsep{2mm}
\xymatrix{
0\ar[r] & zA \ar[r] & A \ar[r] & A/zA \ar[r] & 0,\\
0\ar[r] & (0: z)_A \ar[r] & A \ar[r] & zA \ar[r] & 0,
}
\]
we arrive at  
\[
\begin{split}\Tor^A_i(M,A/zA)&\simeq\Tor^A_{i-1}(M,zA)
\\
&\simeq\Tor^A_{i-2}(M,(0: z)_A);
\end{split}\]
by the induction hypothesis jointly with the   facts that $i-2\ge 1$  and that $(0: z)_A$ is annihilated by $z$, we conclude that $\Tor_i^A(M,A/zA)=0$. Consequently,  for any free $(A/zA)$-module $F$ we have 
\[\tag{$\S$}\Tor_i^A(M,F)=0.
\] 
Now, for an  $(A/zA)$-module $N$, let $0\to U\to F\to N \to 0$ be an exact sequence of $(A/zA)$-modules, where $F$ is free. By the induction hypothesis and $(\S)$, we   arrive at  an exact sequence
\[
\underbrace{\Tor^A_i(M,F)}_{=0}  \to \Tor^A_i(M,N) \to \underbrace{\Tor^A_{i-1}(M,U)}_{=0}.
\]
Hence $\Tor^A_i(M,N)=0$, as desired.
\end{proof}
 
\begin{proof}[Proof of Theorem \ref{thm_fibercriterion_strong}]
Roughly speaking, we shall cut $\text{Spec} \, A$ by ``hypersurfaces" determined by elements of $R$. The key point is to ensure that the hypotheses for $R$ pass to the quotient rings. 

We note that condition (F1) is equivalent to 
\begin{enumerate}[\quad \rm (F1')]
 \item for every ideal $\lfr$ of $R$, the natural arrow  ${\mathfrak{l}}A\oti_A M\to {\lfr}M$ is bijective.
\end{enumerate}
Given condition (F1), condition (F3b) can be replaced by
\begin{enumerate}[\quad \rm (F3b')]
 \item for every ideal $\pp \in \Spec\,R$, the map $t(A/\pp A)\oti_A M \to (A/\pp A)\oti_A M$ induced by the inclusion $t(A/\pp A) \to A/\pp A$, is injective.
\end{enumerate}
For this, look at the exact sequence
\[
0 \to t(A/\pp A) \to A/\pp A \to (A/\pp A)_{\tf} \to 0.
\]
We have an induced exact sequence
\[
\underbrace{\Tor^A_1(A/\pp A,M)}_{=0} \to \Tor^A_1((A/\pp A)_{\tf},M) \to t(A/\pp A)\oti_A M \to (A/\pp A)\oti_A M.
\]
Hence (F3b) can be replaced by (F3b').

\noindent\textit{Step 1} (Passage to quotient rings). We claim that the conditions (F1), (F2), (F3a), (F3b') pass to quotients of $R$, namely for any ideal ${\hh}$ of $R$, the corresponding conditions hold for $R/{\mathfrak{h}}, A/{\mathfrak{h}}A$ and $M/{\mathfrak{h}}M$. In detail, this means:
\begin{enumerate}[\quad \rm (F1{$\hh$})]
\item For every ideal ${\mathfrak{l}}\supseteq {\mathfrak{h}}$ of $R$, we have $\Tor^{A/{\mathfrak{h}}A}_1(A/{\mathfrak l}A,M/{\mathfrak{h}}M)=0$;
 \item  The $\left((A/{\mathfrak{h}}A)\oti_{R/{\mathfrak{h}}} \bm k(\qq)\right)$-module $M/{\mathfrak{h}}M\oti_{R/{\mathfrak{h}}} \bm k(\qq)$ is flat for every $\qq\in \Spec(R/{\mathfrak{h}})$;
\end{enumerate}
and  either (F3a{$\hh$}) $A/{\mathfrak{h}}A$ is a noetherian ring, or (F3b'{$\hh$}) $t(\ovl{A}/\qq \ovl{A})\oti_{\ovl{A}} \ovl{M} \to (\ovl{A}/\qq \ovl{A})\oti_{\ovl{A}} \ovl{M}$ is injective for any ideal $\qq\in \Spec(\ovl{R})$. Here $\ovl{R}=R/\hh,\ovl{A}=A/{\hh}A$ and likewise $\ovl{M}=M/{\hh}M$.

Indeed, (F1${\hh}$) is equivalent to 
$$
({\mathfrak{l}}A/{\mathfrak{h}}A) \oti_{A/{\mathfrak{h}}A} M/{\mathfrak{h}}M \simeq {\lfr}M/{\mathfrak{h}}M.
$$
We have $({\mathfrak{l}}A/{\mathfrak{h}}A) \oti_{A/{\mathfrak{h}}A} M/{\mathfrak{h}}M \simeq ({\mathfrak{l}}A/{\mathfrak{h}}A)\oti_A M$. The exact sequence $0\to {\hh}A \to {\mathfrak{l}}A \to {\mathfrak{l}}A/{\mathfrak{h}}A\to 0$ yields a diagram
\[
\begin{xymatrix}
{ &{\mathfrak{h}}A \oti_A M \ar[r] \ar[d]^{\simeq} & {\mathfrak{l}}A\oti_A M \ar[r]\ar[d]^{\simeq} & ({\mathfrak{l}}A/{\mathfrak{h}}A) \oti_A M \ar[r] \ar[d]^{\alpha} & 0\\
0\ar[r] &{\mathfrak{h}}M \ar[r] & {\mathfrak{l}}M \ar[r] & {\mathfrak{l}}M/{\mathfrak{h}}M \ar[r] & 0.
}
\end{xymatrix}
\]
By the Snake Lemma and  (F1), we conclude that $\alpha: ({\mathfrak{l}}A/{\mathfrak{h}}A)\oti_A M\to {\mathfrak{l}}M/{\mathfrak{h}}M$ is an isomorphism.

For any $\qq \in \Spec(R/{\mathfrak{h}})$, we can write $\qq=\pp/{\mathfrak{h}}$, where $\pp\in \Spec\,R$ is a prime ideal containing ${\hh}$. We note that
\begin{align*}
 M/{\mathfrak{h}}M\oti_{R/{\mathfrak{h}}} \bm k(\qq) \simeq M/{\mathfrak{h}}M\oti_{R/{\mathfrak{h}}} (R/\pp)_\pp \simeq M\oti_R {\boldsymbol k}(\pp).
\end{align*}
Thus (F2{$\hh$}) follows from (F2). Clearly (F3a{$\hh$}) follows from (F3a). For (F3b'{$\hh$}), letting $\qq=\pp/{\hh}$ as above, the map
\[
t(\ovl{A}/\qq \ovl{A})\oti_{\ovl{A}} \ovl{M} \to (\ovl{A}/\qq \ovl{A})\oti_{\ovl{A}} \ovl{M}
\]
is nothing but
\[
t(A/\pp A) \oti_A M \to (A/\pp A)\oti_A M.
\]

\noindent\textit{Step 2}. We prove by noetherian induction that the hypotheses (F1), (F2), and either (F3a) or (F3b), imply that $M$ is a flat $A$-module.

We first reduce to \textit{the case where $R$ is a domain}. For this, note that for any non-zero element $z\in R$ and any $(A/zA)$-module $N$, we have $\Tor^A_1(M,N)=0$. Indeed, as $z\neq 0$, by Step 1 and the hypothesis of the noetherian induction, $M/zM$ is a flat $(A/zA)$-module. By (F1), $\Tor^A_1(A/zA,M)=0$. Hence Lemma \ref{lem_reduction_to_domains} implies that $\Tor^A_1(M,N)=0$ for any $(A/zA)$-module $N$.
If $R$ is not a domain, then $yz=0$ for some non-zero elements $y,z\in R$. Using the exact sequence
\[
0\to yN \to N \to N/yN \to 0,
\]
and the fact that $yN, N/yN$ are annihilated by non-zero elements $z,y$, respectively, we conclude that $\Tor^A_1(M,N)=0$ for any $A$-module $N$. Thus $M$ is a flat $A$-module if $R$ is not a domain.

Now assume that $R$ is a domain. Let $K$ be its quotient field.  To show the equality $\Tor^A_1(M,N)=0$, it suffices to treat $t(N)$ and $N_{\tf}$ in place of $N$, namely we can assume that either  $N=t(N)$ or $t(N)=0$. Consider two cases.

\noindent \textbf{Case 1}. Assume that $N$ is $R$-torsion, $N=t(N)$.  Since Tor commutes with direct limits, we can assume that $N$ is a finitely generated $A$-module. Hence there exists $z\neq 0$,  $zN=0$. This already suffices for the equality $\Tor^A_1(M,N)=0$, thanks to Lemma \ref{lem_reduction_to_domains}.

\noindent \textbf{Case 2}. Assume that $N$ is $R$-torsionfree. We consider two subcases corresponding to assumption (F3a) and (F3b). 

If (F3a) holds, then $A$ is noetherian. For any $z\in R\smallsetminus \{0\}$, we have an exact sequence
\[
0\to N \stackrel {\cdot z}\to N \to N/zN \to 0.
\]
Observe that by the hypotheses of the noetherian induction, the assumptions of Lemma \ref{lem_reduction_to_domains} are fulfilled. Hence $\Tor^A_1(M,N/zN)=\Tor^A_2(M,N/zN)=0$. We get an exact sequence
\[
\underbrace{\Tor^A_2(M,N/zN)}_{=0} \to \Tor^A_1(M,N) \stackrel {\cdot z}\to \Tor^A_1(M,N) \to \underbrace{\Tor^A_1(M,N/zN)}_{=0}.
\]
Thus $\Tor^A_1(M,N) \stackrel {\cdot z}\to \Tor^A_1(M,N)$ is an isomorphism for any $z\in R\smallsetminus \{0\}$. On the other hand, denoting the functor $\bullet\longmapsto\bullet\oti_RK$ by   $(\bullet)_K$,   then $\Tor^A_1(M,N)\oti_R K\simeq \Tor^{A_K}_1(M_K,N_K)=0$, as per (F2), $M_K$ is flat over $A_K$. Hence any element of $\Tor^A_1(M,N)$ is annihilated by \emph{some} $z\in R\smallsetminus \{0\}$. In particular, $\Tor^A_1(M,N)=0$.

If (F3b) holds, then $\Tor^A_1((A/\pp A)_{\tf},M)=0$ for every $\pp \in \Spec\,R$, which implies
 $\Tor^A_1(A_{\tf},M)=0$. 
Let $F\to N$ be a surjection from a free $A$-module. Since $N$ is $R$-torsion, there is an induced map $F_{\tf} \to N$, which remains to be surjective. 
Consider the commutative diagram with exact rows where the vertical arrows are localization maps:
\[
\xymatrix{
0\ar[r] & U \ar[d] \ar[r] & 
F_{\tf}\ar[r]\ar[d] & N \ar[r] \ar[d]  & 0\\
0\ar[r] & U_K \ar[r] & (F_{\tf})_K \ar[r] & N_K \ar[r]  & 0.
}
\]
Notice that $F_{\tf}$ is a direct sum of copies of $A_{\tf}$.
Hence, (F3b) and the long exact sequence of Tor yields
\[
\xymatrix{
 0\ar[r] & \Tor^A_1(M,N) \ar[d]_{\theta} \ar[r]^{\beta} & M\oti_A U \ar[d]^{\alpha}\\
& \Tor^A_1(M,N_K) \ar[r] & M\oti_A U_K\ar[r]^-{\gamma}& M\oti_A (F_{\tf})_K.
}
\]

Since $U$, as an $R$-module, is  torsionfree (being a submodule of $F_{\tf}$), we have an exact sequence
\[
0\to U \to U_K \to U_K/U \to 0.
\]
On the other hand, the module $U_K/U$ is torsion, so Case 1 implies that $\Tor^A_1(M,U_K/U)=0$. Hence the map $\alpha$ of the diagram is injective.
Since $M_K$ is flat, the map
$\gamma: M\oti_A U_K
\to M\oti_A(F_{\tf})_K$ is injective.
 Since $\beta$ is also injective
 this forces $\Tor^A_1(M,N)=0$.
The induction and the proof are completed.
\end{proof}
Finally we present the

\begin{proof}[Proof of Theorem \ref{thm_fibercriterion_prime}] 
We proceed by noetherian induction. Assume that the statement is true for all proper quotient rings of $R$. We prove that it is true for $R$, namely $M$ is a flat $A$-module if $R, A, M$ satisfy the hypotheses (i), (ii), and (iii).

As in the proof of Theorem \ref{thm_fibercriterion_strong}, for any non-zero prime $\pp\in \Spec\,R$, the hypotheses (i)--(iii) pass to the map $R/\pp \to A/\pp A$ and the module $M/\pp M$. Hence by the induction hypothesis, $M/\pp M$ is $(A/\pp A)$-flat for any such $\pp$.

We prove the hypotheses (i)--(iii) imply that 
$\Tor^A_1(M,N)=0$ for any $A$-module $N$ which is annihilated by a non-zero ideal $\g l$ of $R$. 
Since $R$ is noetherian, $\mathfrak l$ contains a finite product of (not necessarily distinct) non-zero prime ideals $\pp_1\cdots \pp_n$. By filtering, it is led to showing  that $\Tor^A_1(N,M)=0$ if $N$ is annihilated by a non-zero prime ideal $\pp\in \Spec\,R$. This follows from Lemma \ref{lem_localcriterion_generalform}, as, by hypothesis, $\Tor^A_1(A/\pp A,M)=0$,  $M/\pp M$ is flat over $A/\pp A$.

Now (i) can be replaced by the stronger condition $\Tor^A_1(A/{\lfr}A,M)=0$ for any non-zero ideal $\lfr\subset R$. Thus we are in the situation of Theorem \ref{thm_fibercriterion_strong}, and so $M$ is a flat $A$-module. This concludes the noetherian induction and the proof. 
\end{proof}
An immediate consequence of Theorem \ref{thm_fibercriterion_prime} is
\begin{cor}
\label{cor_fibercriterion_prime}
Let $R$ be a noetherian ring and  $f:R\to A$ a ring map. Let $M$ be an $A$-module. Then $M$ is a flat $A$-module if and only if for every  $\pp \in \Spec\,R$, the following conditions hold:
\begin{enumerate}[\quad \rm (i)]
 \item \textup{(Tor vanishing over primes)}  $\Tor^A_1(A/\pp A,M)=\Tor^A_1((A/\pp A)_{\tf},M)=0$;
 \item \textup{(Fiber flatness)} The $\left(A\oti_R {\boldsymbol k}(\pp)\right)$-module $M\oti_R {\boldsymbol k}(\pp)$ is  flat.
\end{enumerate}
\end{cor}

\section{Fiber criteria for purity of maps}
\label{sect_purity}

Let us now apply the our results   to study pure submodules. In this section, $A$ stands for a ring.

\begin{dfn}[Pure morphisms]An arrow  $\varphi:M\to N$ of $A$--modules is said to be pure if for any $A$--module $E$, the induced arrow $\varphi\oti_AE:M\oti_AE\to N\oti_AE$ is injective.  
\end{dfn} 

A simple application of the long exact sequence  for Tor proves the following: 
\begin{lem}
\label{lem_pure_equals_flatcoker}
Let $\varphi:M\to N$ be an \emph{injective} arrow of $A$--modules and assume that $N$ is \emph{flat}. Then a necessary and sufficient condition for $\varphi$ to be pure is that $\mathrm{Coker}(\varphi)$ is flat.   \qed
\end{lem}

We also have the following criterion for pure maps.


\begin{prop}
[Fiber criterion for pure maps] 
\label{prop_fiber_puremap} Let $R$ be a noetherian ring and $R\to A$ a ring homomorphism. Let $\varphi:M\to N$ be a map of  $A$--modules, where $N$ is flat over $A$. Then,  a necessary and sufficient condition for $\varphi$ to be a pure map of $A$-modules, is that for all $\pp\in\Spec\,R$:
\begin{enumerate}[\quad \rm (i)]
\item The map $\varphi \oti_A (A/\pp A): M\oti_A (A/\pp A)\to N\oti_A (A/\pp A)$ is injective;
\item The map $\varphi \oti_A (A/\pp A)_{\tf}: M\oti_A (A/\pp A)_{\tf}\to N\oti_A (A/\pp A)_{\tf}$ is injective;
\item The map  $\varphi\oti_R {\boldsymbol k}(\pp):M\oti_R {\boldsymbol k}(\pp)\to N\oti_R {\boldsymbol k}(\pp)$ of $(A\oti_R {\boldsymbol k}(\pp))$-modules is pure. 
\end{enumerate}
\end{prop}

\begin{rmk}
If $A$ is $R$-flat, then    hypothesis (ii) in Proposition \ref{prop_fiber_puremap} is already assured by (i) and is therefore superfluous.
\end{rmk}

\begin{proof}
 Necessity is clear and we move on to  establish  sufficiency. Let $W=\Img \varphi$ and $P=\Coker \varphi$. These modules fit into the exact sequence $0\to W \to N \to P \to 0$.

\noindent\textit{Step 1}. The hypothesis implies that for each $\pp \in \Spec\,R$, the map $W/\pp W\to N/\pp N$ is injective. Flatness of $N$ yields an exact sequence
\[
\underbrace{\Tor^A_1(A/\pp A,N)}_{=0} \to \Tor^A_1(A/\pp A,P) \to W/\pp W \to N/\pp N
\]
so that  $\Tor^A_1(A/\pp A,P)=0$.
Similarly, $\Tor^A_1((A/\pp A)_{\tf},P)=0$. 
By (iii), the map of $A\oti_R\boldsymbol k(\g p)$-modules $M\oti_R {\boldsymbol k}(\pp)\to N\oti_R {\boldsymbol k}(\pp)$ is pure with cokernel $P\oti_R {\boldsymbol k}(\pp)$, and since $N\oti_R {\boldsymbol k}(\pp)$ is flat over $A\oti_R {\boldsymbol k}(\pp)$, Lemma \ref{lem_pure_equals_flatcoker} says that $P\oti_R {\boldsymbol k}(\pp)$ is a flat module over $A\oti_R {\boldsymbol k}(\pp)$. All conditions required to apply   Theorem \ref{thm_fibercriterion_prime} are in place  and we  conclude that $P$ is a flat $A$-module.

Since $0\to W \to N \to P \to 0$ is an exact sequence, and $N$ and $P$ are flat $A$-modules, so is $W$; moreover, $W\to N$ is a pure map. It remains to show that $M=W$, namely, that  $U:=\Ker \varphi$ is the zero module.

\noindent\textit{Step 2}.  We proceed by noetherian induction to show that $\varphi$ is a pure map (equivalently, is injective). If $R$ is a domain, then $\pp=(0)\in \Spec\,R$, and by condition (i), $U=0$. Assume that $R$ is not a domain. As $R$ is noetherian, there exist finitely many prime ideals $\pp_1,\ldots,\pp_n$ of $R$ whose product is zero. Now $n\ge 2$, and we choose $n$ to be smallest possible.

Denote $\pp_2\cdots \pp_n$ by $\lfr$; this is a non-zero ideal. It is not hard to see that the hypotheses of Proposition \ref{prop_fiber_puremap} pass to the maps $R/{\mathfrak l} \to A/{\mathfrak l}A$ and  $M/{\mathfrak l}M \to N/{\mathfrak l}N$. Since $\lfr\neq (0)$, by the induction hypothesis, $M/{\mathfrak l}M \to N/{\mathfrak l}N$ is injective. Hence $U=\Ker \varphi$ is contained in ${\lfr}M$. Denoting $\pp=\pp_1$, this implies that $\pp U=0$ as $\pp {\lfr}=0$.

Consider the exact sequence
\[
0\to U \to M \to W \to 0.
\]
Since $W$ is a flat $A$-module, we get an exact sequence
\[\underbrace{
\Tor^A_1(A/\pp A,W)}_{=0} \to U/\pp U \to M/\pp M \to W/\pp W\to 0. 
\]
The injective map $M/\pp M \to N/\pp N$ factors through $M/\pp M \to W/\pp W$, hence the last map is injective. Therefore $U/\pp U=0$, and thus $U=\pp U=0$, as desired.
\end{proof}

\begin{cor}
\label{cor_fibercrit_puremap}
Let $R$ be a noetherian ring and $R\to A$ be a ring map such that for every $\pp \in \Spec\,R$, $A/\pp A$ is a torsionfree $(R/\pp)$-module. Let $\varphi: M\to N$ be a linear map of $A$-modules where $N$ is flat over $A$. Then the following are equivalent:
\begin{enumerate}[\quad \rm (1)]
 \item The map of $A$-modules $\varphi$ is   pure;
 \item For every $\pp \in \Spec\,R$, $M/\pp M$ is a torsionfree $(R/\pp)$-module, and $\varphi\oti_R {\boldsymbol k}(\pp): M\oti_R {\boldsymbol k}(\pp) \to N\oti_R {\boldsymbol k}(\pp)$ is a pure map of $(A\oti_R {\boldsymbol k}(\pp))$-modules.
\end{enumerate}
\end{cor}
\begin{proof}
(1) $\Longrightarrow$ (2): If $\varphi$ is pure, then clearly for any $\pp \in \Spec\,R$, $\varphi\oti_R {\boldsymbol k}(\pp)$ is a pure map of $(A\oti_R {\boldsymbol k}(\pp))$-modules.
Consider the commutative diagram with the vertical maps induced by localization
\[
\xymatrix{
M\oti_R (R/\pp) \ar[r]^{\alpha} \ar[d]_{\pi_M} & N\oti_R (R/\pp) \ar[d]^{\pi_N}\\
M\oti_R {\boldsymbol k}(\pp) \ar[r]_{\beta} & N\oti_R {\boldsymbol k}(\pp).
}
\]
Since $\varphi$ is pure, $\alpha$ is injective. The map $\pi_N$ can be identified with $N\oti_A (A/\pp A) \to N\oti_A (A\oti_R {\boldsymbol k}(\pp))$. Since $A/\pp A$ is $(R/\pp)$-torsionfree, $A/\pp A \to A\oti_R {\boldsymbol k}(\pp)$ is injective. As $N$ is flat over $A$, we infer that $\pi_N$ is injective. Hence $\pi_N \circ \alpha$, and hence $\pi_M$ is injective. This means that $M/\pp M$ is a torsionfree $(R/\pp)$-module.

(2) $\Longrightarrow$ (1): As $A/\pp A$ is torsionfree over $R/\pp$, the conditions (i) and (ii) of Proposition \ref{prop_fiber_puremap} coincide. It suffices to check that $M\oti_R (R/\pp) \xrightarrow{\varphi \oti_R (R/\pp)} N\oti_R (R/\pp)$ is injective. Again consider the above diagram. Since $M/\pp M$ is torsionfree over $R/\pp$, the map $\pi_M$ is injective. By the hypothesis, $\beta$ is injective. Hence $\alpha=\varphi \oti_R (R/\pp)$ is also injective, as desired. 
\end{proof}

\begin{cor}
\label{cor_fiberpurecriterion_onebase}
Let $R$ be a noetherian ring. Let $\varphi: M\to N$ be a linear map of $R$-modules where $N$ is flat. Then the following are equivalent: 
\begin{enumerate}[\quad \rm (1)]
 \item The map $\varphi$ is pure;
 \item For every $\pp \in \Spec\,R$, the $(R/\pp)$-module $M/\pp M$ is   torsionfree  and the map $\varphi\oti_R {\boldsymbol k}(\pp): M\oti_R {\boldsymbol k}(\pp) \to N\oti_R {\boldsymbol k}(\pp)$ is injective.
\end{enumerate}
\end{cor}
\begin{proof}
This follows from Corollary \ref{cor_fibercrit_puremap} by setting $A=R$. Note that ${\boldsymbol k}(\pp)$ is a field, and for maps between vector spaces over a field, purity is the same as injectivity.
\end{proof}

\section{Counterexamples}
\label{sect_counterexamples}

In view of Theorems \ref{thm_fibercriterion_prime} and \ref{thm_fibercriterion_strong}, it is     natural to ask the following
\begin{question}
\label{question_fibercrit}
Let $R$ be noetherian and 
$f:R\to A$ be a   homomorphism. Let $M$ be an $A$-module such that the following conditions are simultaneously fulfilled: 
\begin{enumerate}[\quad \rm (1)]
\item As an $R$-module, $M$ is  flat;
\item for every $\pp \in \Spec\,R$, $\Tor^A_1(A/\pp A,M)=0$;
 \item for every $\pp \in \Spec\,R$, $M\oti_R {\boldsymbol k}(\pp)$ is a flat $\left(A\oti_R {\boldsymbol k}(\pp)\right)$-module.
 \end{enumerate}
 Is it true that $M$ is a flat $A$-module?
 (In other words, we exchange hypothesis (i) of Theorem \ref{thm_fibercriterion_prime} wit the hypothesis of $M$ being flat over $R$.)
\end{question}

The next example shows that this is false in general. 

\begin{ex}
Let $k$ be a field, and $R=k\llbracket t\rrbracket$ the power series ring in one variable $t$. In the ring $S=R[x_1,x_2,\ldots]$ consider the following ideals  
\[
\mathfrak H=(t^ix_i-t^{i+1}x_{i+1},x_i-t^2x_{i+2}: i\ge 1)\subset\mathfrak I=(x_i-tx_{i+1}:i\ge 1).
\]
Let $A$ be the quotient $S/\mathfrak H$ and  $M$ be the $A$-module $S/\mathfrak I$ (recall that   $\mathfrak H\subset\mathfrak I$). We claim that:  
\begin{enumerate}
[\quad\rm(1)]
\item As an $R$-module, $M$ is   flat;
\item for every $\pp \in \Spec\,R$, we have $\Tor^A_1(A/\pp A,M)=0$;
 \item for every $\pp \in \Spec\,R$, $M\oti_R k(\pp)$ is a flat $\left(A\oti_R k(\pp)\right)$-module;
 \item $M$ is not a flat $A$-module.
\end{enumerate}

As explained in   Remark \ref{rmk_torsionfree_hypothesis}, condition (2) above is a consequence of condition (1), so we only deal with (1), (3) and (4). Letting $K$ be the field of fractions of $R$, it suffices to establish the following.
\begin{enumerate}[\quad \rm (i)]
\item The $R$-module $M$ is torsionfree, or alternatively, $\Tor^R_1(R/(t),M)=0$.
 \item $M\oti_R K$ is a flat $(A\oti_R K)$-module and $M/tM$ is a flat $(A/tA)$-module.
 \item $\Tor^A_1(A_{\tf}, M)\neq 0$.
 \end{enumerate}
 For this, we have the following claims:
\begin{enumerate}[\quad\rm (a)]
\item In $S$, we have $(\mathfrak I:t)_S=\mathfrak I$.
\item $t(A)=\Ker(A\to A\oti_R K)=\mathfrak I/\mathfrak H$, so $M=A/t(A)=A_{\tf}$ and $M\oti_R K=A\oti_R K$.
\item $\mathfrak I+(t)=\mathfrak H+(t)$, so $M/tM=A/tA$.
\item $x_1-tx_2\in\mathfrak  I\smallsetminus (\mathfrak I^2+ {\mathfrak H})$, so 
\[\begin{split}
\Tor^A_1(A_{\tf}, M)&=\Tor^A_1(A/t(A), A/t(A))
\\&=t(A)/t^2(A)\\&=\mathfrak I/(\mathfrak I^2+\mathfrak H)\\&\neq 0.\end{split}\]
 \end{enumerate}
 
Equip $S$ with the grading such that elements of $R$ have degree zero and  $\deg x_i=1$ for all $i$. Then $\mathfrak I$ and $\mathfrak H$ are homogeneous ideals generated by elements of degree 1.
 
(a) Assume that $f\in S$ is a homogeneous element in the above grading, and $tf\in \mathfrak I$. Then $f\in B=R[x_1,\ldots,x_d]$ for some finite $d$. Now $tf\in\mathfrak  I$, so increasing $d$ if necessary, we may assume $tf \in (x_1-tx_2,\ldots,x_{d-1}-tx_d)\subset B$. To show that $f\in\mathfrak  I$, as $x_i-tx_{i+1}\in\mathfrak  I$, we can assume that $f=ax_d^n$ for some $a\in R$ and some $n\ge 0$. Now $tf=tax_d^n\in (x_1-tx_2,\ldots,x_{d-1}-tx_d)$. Setting $x_i=t^{d-i}$ for $i=1,\ldots,d$, we get $ta=0$, hence $a=0$. Thus $f\in \mathfrak I$.

(b) First $\mathfrak I/\mathfrak H\subset t(A)$, since $t^i(x_i-tx_{i+1})\in \mathfrak H$ for all $i\ge 1$.

Conversely, take $f\in S$ homogeneous such that $bf\in\mathfrak  H$ for some $b\in R\smallsetminus \{0\}$. We have to show that $f\in \mathfrak I$. Again $f\in B=R[x_1,\ldots,x_d]$ for some finite $d$. Now $bf\in\mathfrak  H$, so increasing $d$ if necessary, we may assume 
$$
bf \in (t^ix_i-t^{i+1}x_{i+1},x_j-t^2x_{j+2}:1\le i\le d-1, 1\le j\le d-2)=\mathfrak H'\subset B.
$$
We can write $f=f_1+f_2$, where $f_1\in (x_i-tx_{i+1}: 1\le i\le d-1) \subset \mathfrak I$, and $f_2=ax_d^n$ for some $a\in R, n\ge 0$. Clearly $t^{d-1}f_1\in\mathfrak  H'$, so $bt^{d-1}f_2\in\mathfrak  H'$. Now $bt^{d-1}f_2=abt^{d-1}x_d^n\in \mathfrak H'\subset B$. Setting $x_i=t^{d-i}$ for $i=1,\ldots,d$, we get $t^{d-1}ba=0$, hence $a=0$. Thus $f=f_1\in\mathfrak  I$.

Since $M=A/t(A)$, clearly $M\oti_R K=A\oti_R K$.

(c) We have $\mathfrak I+(t)=\mathfrak H+(t)=(t,x_i: i\ge 1)$, so $M/tM=A/tA\simeq k$. 

(d) Assume that $x_1-tx_2\in {\mathfrak I}^2+{\mathfrak H}$. Since $x_1-tx_2$ is homogeneous of degree 1, we get $x_1-tx_2\in \mathfrak H$. Thus for some large $d$,
\[
x_1-tx_2=\sum_{i=1}^{d-1} a_i(t^ix_i-t^{i+1}x_{i+1})+\sum_{j=1}^{d-2}b_j(x_j-t^2x_{j+2}), \quad a_i,b_j\in R.
\]
Denote $y_i=x_i-tx_{i+1}$, then $x_j-t^2x_{j+2}=y_j+ty_{j+1}$. Thus
\begin{align*}
y_1 &=\sum_{i=1}^{d-1} a_it^iy_i+\sum_{j=1}^{d-2}b_j(x_j-t^2x_{j+2})=
\sum_{i=1}^{d-1} a_it^iy_i+\sum_{j=1}^{d-2}b_j(y_j+ty_{j+1})\\
    &=\sum_{i=1}^{d-1}(a_it^i+b_{i-1}t+b_i)y_i \quad \textnormal{(where $b_i=0$ for $i\in \{0,d-1\}$)}.
\end{align*}
Since $y_1,\ldots,y_{d-1}$ are algebraically independent over $R$, this implies
\[
\begin{cases}
 a_1t+b_1 &=1,\\
 a_it^i+b_{i-1}t+b_i&=0, \quad \text{for $2\le i\le d-1$}.
\end{cases}
\]
Since $b_{d-1}=0$, by reverse induction $b_i\in (t^i)$ for each $1\le i\le d-1$. But then $1=a_1t+b_1\in (t)$, a contradiction. Hence $x_1-tx_2\notin (\mathfrak I^2+\mathfrak H)$.
\end{ex}

Similarly, we may ask if the torsionfree condition in hypothesis (2) of \ref{cor_fiberpurecriterion_onebase} banana can be dropped. 
Using an example due to Christensen--Iyengar--Marley \cite{CIM}, we show that the answer is again negative.
\begin{lem}
\label{lem_CIM}
For any $d\ge 2$, there exist a complete noetherian local ring $(R,\mm,k)$ and a (non-finite) $R$-module $P$ satisfying simultaneously the following conditions:
\[
\begin{cases}
\Tor^R_i(k,P)= 0, \quad \text{for $i\le d-1$},\\
\Tor^R_i(k,P)\neq 0, \quad \text{for $i\ge d$},\\
\Tor^R_i({\boldsymbol k}(\pp),P)=0, \quad  \text{for every $i$ and every $\pp \in \Spec\,R\smallsetminus \mm$}.
\end{cases}
\]
Moreover, $P$ is not  a flat $R$-module (and moreover, does not have a finite flat dimension). 
\end{lem}
\begin{proof}
We use \cite[Example 4.2]{CIM}. Take $R=k\llbracket x_1,\ldots,x_{d+1}\rrbracket/(x_1^2,x_1x_2,\ldots,x_1x_{d+1})$, $\mm=(x_1,\ldots,x_{d+1})$. Let $P=H^d_\mm(R/(x_1))$, the $d$-th local cohomology of $R/(x_1)$ with support at the maximal ideal $\mm$.  It is proved in \emph{ibid.} that $R$ and $P$ have the required properties.
\end{proof}

\begin{cor}
There exist a complete noetherian local ring $(R,\mm)$ and an injective map $g: M\to N$ of $R$-modules which is \textbf{not} pure and  such that:
\begin{enumerate}[\quad \rm (1)]
 \item The base change $M\oti_R {\boldsymbol k}(\pp) \to N\oti_R {\boldsymbol k}(\pp)$ is injective for every $\pp \in \Spec\,R$;
\item the module $N$ is   $R$-flat.
\end{enumerate}
\end{cor}
\begin{proof}
Let $(R,\mm)$ and $P$ be as in Lemma \ref{lem_CIM}. Choose a free $R$-module $N$ which surjects onto $P$, and let $M$ be the kernel of this surjection.
Then $\Tor^R_1({\boldsymbol k}(\pp),P)=0$ implies that $M\oti_R {\boldsymbol k}(\pp) \to N\oti_R {\boldsymbol k}(\pp)$ is injective for every $\pp \in \Spec\,R$. But by our choice, $P$ is not flat, hence Lemma \ref{lem_pure_equals_flatcoker} implies that  $g$ fails to be pure.  
\end{proof}

\section{Applications to group schemes}
\label{sect_applications}

Using the material developed in the previous sections, we establish     criteria for morphisms of affine group schemes   to be faithfully flat. This is in order not only because affine group schemes which fail to be of finite type are important objects in the Tannakian theory, but also because, for group schemes, there are rather strong results such as the ``faithful flatness of subalgebras'' \cite[Ch. 14]{Wat}. 

In this section,    $R$ stands for  a noetherian ring.

\subsection{Fiber criteria for group scheme homomorphisms}\label{04.12.2023--1}
 
The following result is a consequence of the fundamental theorem about faithful flatness of Hopf subalgebras \cite[14.1, Theorem]{Wat} and our previous findings. (Recall that a morphism of group schemes is faithfully flat if the underlying morphism of schemes is faithfully flat.)

\begin{thm}[Faithful flatness of pure subalgebras]\label{pure_flat} Let $f: G'\longrightarrow G$ be a homomorphism of flat affine group schemes over $R$. The following are equivalent:
\begin{enumerate}
[\quad\rm(1)]
\item The morphism $f$ is faithfully flat;
\item For each $\mathfrak p\in \text{\rm Spec}\, R$ the morphism 
$f\otimes_R\boldsymbol k(\g p)$ is faithfully flat;
\item The morphism induced on $R$-algebras $f^\sharp:R[G]\to R[G']$ is  pure. 
\end{enumerate} \end{thm}
 \begin{proof} 
 The implication (1) $\Longrightarrow$ (2) is well-known. 
That  (2) $\Longrightarrow$ (3) follows from Corollary \ref{cor_fiberpurecriterion_onebase} since  $R[G]$ and $R[G']$ are $R$-flat. 
Finally, let us assume that (3) holds. It then follows that, for any $\g p\in\mathrm{\rm Spec}\,R$, the morphism of $\boldsymbol k(\g p)$-modules   $f^\sharp\otimes_R\boldsymbol k(\g p)$ is injective; this, when allied with \cite[14.1, Theorem]{Wat}, shows that $f^\sharp\otimes_R\boldsymbol k(\g p)$ 
is faithfully flat and we conclude that (1) is true by Corollary \ref{cor_fiber_faithfullyflat}. 
\end{proof}

\subsection{Flatness criteria involving categories of representations}\label{12.12.2023--1}

Let $\mathbf{Rep}_{\mathbf f}\,G$ stand for the category of representations of 
$G$ whose underlying $R$-module is in addition finitely generated, see Section \ref{prelimina_group_schemes}. Given another flat affine group scheme $G'$ and a morphism   $f:G'\to G$, we shall apply our flatness criteria to relate   properties of  the  functor \[f^*:\mathbf{Rep}_{\mathbf f}\,G\to\mathbf{Rep}_{\mathbf f}\,G' 
\] 
to properties of  $f$. This shall generalize    
\cite[Prop.~2.21]{DM82}.

Before proceeding, let us explain some categorical terminology. Recall that a functor   $F:\mathcal C\to \mathcal C'$ is fully faithful if the map $\mathrm{Hom}_{\mathcal C}(V,W)\to\mathrm{Hom}_{\mathcal C'}(FV,FW)$ is bijective for all objects $V,W$ of $\mathcal C$. Now, if $F$ is fully faithful,  $\mathcal C$ and $\mathcal C'$ are abelian, and $F$ is exact \cite[VIII.3, p. 197]{maclane98}, we say that its (essential) image is {\it stable under subobjects} if  for  every $V\in\mathcal C$ and every subobject $W'\to  F(V)$, there exists a subobject $W\to  V$ and an isomorphism $F(W)\simeq W'$. (This terminology is not widespread.)

\begin{thm}\label{surjective_criterion_group_hom}
  The morphism  $f$ is  faithfully flat if and only if the functor $f^*$ is fully faithful  and its essential  image  is stable under subobjects.   
\end{thm}
\begin{proof}
According to Theorem \ref{pure_flat}, $f$ is   faithfully flat if and only if   $f^\sharp:R[G]\to R[G']$ is pure. Therefore this theorem is a consequence of the next proposition for coalgebra  morphisms.
\end{proof}

\begin{thm}\label{flat_criterion_coalgebra_hom}
Let $\varphi: L\to L'$ be a homomorphism of flat $R$-coalgebras. Then $\varphi$ is pure if and only if  the functor $\varphi_*:\Comodf L \to\Comodf {L'}$  is fully faithful  and its essential image is stable under subobjects.  
\end{thm}
\begin{proof}

$(\Rightarrow)$. The argument here is essentially the same as in   \cite[Prop.~3.2.1~(ii)]{DH18}; we reproduce it for the convenience of the reader.  
Assume that $\varphi$ is   pure  and let 
$\rho:M\to M\otimes L$ and $\sigma:N\to N\otimes L$ define $L$-comodules.
 Let $h:M\to N$ be an $R$-linear map which is also an arrow in   $\Comodf {L'}$.
This means that the diagram 
$$\xymatrix{
M\ar[r]^-{\rho}\ar[d]_h&M\otimes L\ar[rr]^{\mathrm{id}_M\otimes \varphi}
&& M\otimes L'\ar[d]^{h\otimes\mathrm{id}_{L'}}\\
N\ar[r]_-{\sigma}&N\otimes L\ar[rr]_{\mathrm{id}_N\otimes \varphi}&& N\oti L' 
}
$$
commutes, since the horizontal compositions define the coactions of $L'$. Now 
\[\xymatrix{
M\otimes L\ar[rr]^-{\mathrm{id}_M\otimes \varphi}
\ar[d]_{h\otimes\mathrm{id}_{L}}
&& M\otimes L'\ar[d]^{h\otimes\mathrm{id}_{L'}}
\\
N\oti L\ar[rr]_-{\mathrm{id}_N\otimes \varphi}&& N\oti L',
}\]
is tautologically commutative and the injectivity of $\mathrm{id}_N\otimes\varphi$ plus a diagram chase assure  that $(h\otimes\mathrm{id}_L)\circ\rho=\sigma\circ h$. Now, let us perform the following modifications on the above notations to deal with stability under subobjects. Assume that $h$ is injective, that $M$ only carries  a structure of $L'$-comodule ---  we denote by $\rho':M\to M\otimes L'$ --- and that $h$ is a map of $L'$-comodules. We want to show that $\rho'(M)\subset M\otimes L$.  Let $\pi:L'\to C$ stand for the cokernel of $\varphi$; we know that $C$ is $R$-flat (cf. Lemma \ref{lem_pure_equals_flatcoker}). A simple diagram chase in the commutative diagram 
\[
\xymatrix{
M\ar@{^{(}->}[d]_h\ar[rrr]^{\rho'}&&&M\otimes L'\ar@{^{(}->}[d]^{h\otimes\mathrm{id}_{L'}}\ar[rr]^{\mathrm{id}_M\otimes\pi} && M\otimes C\ar@{^{(}->}[d]^{h\otimes\mathrm{id}_{C}}
\\
N\ar[r]_-{\sigma}&N\otimes L\ar[rr]_-{\mathrm{id}_N\otimes \varphi}&&N\otimes L'\ar[rr]_{\mathrm{id}_N\otimes\pi} && N\otimes C
}
\]
shows that $(\mathrm{id}_M\otimes \pi)\circ\rho'=0$ and hence $\rho'$ factors as $M\stackrel{\rho}\to M\otimes L\stackrel{\mathrm{id}\otimes \varphi}\to M\otimes L'$. It is obvious that $\rho$ then defines a coaction of $L$. 


$(\Leftarrow)$. According to  Corollary \ref{cor_fiberpurecriterion_onebase}, it is enough to show that   for all  $\mathfrak p\in\Spec\,R$, the $R/\g p$-linear map  $\varphi\oti R/\mathfrak p$ is injective, as this guarantees that $\varphi\oti  \boldsymbol k( \mathfrak p)$ is likewise injective. 
On the other hand,   $\Comodf{ L\oti R/\pp}$  and    $\Comodf{L'\oti R/\pp}$ are  the full subcategories of $\Comodf  L $, respectively $\Comodf{L'}$, consisting of objects  annihilated by $\pp$. Therefore it is enough to give a proof in the case where  $R$ is an integral domain. 

We use the induction of the dimension of $R$; when $\dim R=0$, this is already well-documented in the literature, see the proof of \cite[Prop.~2.21]{DM82} as well as \cite[Lemma 1.2 and Remark 1.3]{Hai16}. Now, if $\g p\not=0$ is a prime, then the induction hypothesis allows us to say that $\varphi\otimes R/\g p$ is pure and in particular injective. We then only need to show that   $\varphi\oti_RK$ is injective,  where $K$ is the quotient field of $R$.  For that, we shall show that 
\[ (\varphi\oti K)_*:\Comodf{ L\otimes K }\to\Comodf{ L'\otimes K }\] is full and stable under subobjects.

Thus, let $\rho:V\to V\otimes_K(L\otimes K)$ define a comodule for $
L\otimes  K$ and write $\rho':V\to V\otimes_K(L'\oti K)$ for the coaction of $L'\otimes K$ induced by $\varphi$. 
Let   $W\subset V$ be a subspace which is an $L'\otimes K$-subcomodule of $V$: we have $\rho'(W)\subset W\oti_K(L'\oti K)$.    

Let $s\le r$ be the dimensions of $W$ and $V$ over $K$, and let $\{v_i\}_{i=1}^r$ be a basis of $V$ such that $\{v_i\}_{i=1}^s$ is a basis of $W$. 
According to Lemma \ref{lem_serre} below, there exists a finite $R$-submodule   $\mathcal V$ of (the $R$-module) $V$ containing $\{v_i\}_{i=1}^r$, and such that 
\[\begin{split}
\rho(\mathcal V)&\subset \mathcal V\oti_RL\\
&\subset  V\oti_RL
\\&=V\otimes_K(L\oti K).
\end{split}\]
Put differently, $\rho$ defines a structure of $L$-comodule on $\mathcal V$.

 Let $\mathcal W:=\mathcal V\cap W$. 
 Then, because of \cite[Theorem 7.4(i)]{Mat89} and flatneess of $L'$ over $R$, we conclude that   $\rho'(\mathcal W)$, which is a subset of $W\otimes_K(L'\otimes K)=W\oti_RL'$, is contained in  $\mathcal W\oti_RL'
$; we conclude that $\mathcal W$ is an $L'$-subcomodule of $\mathcal V$. The assumption on the functor $\varphi_*$ now assures that  $\mathcal W$ is 
an $L$-subcomodule of $\mathcal V$, i.e. $\mathcal W$ is
sent by $\rho:\mathcal V\to \mathcal V\oti L$ into $\mathcal W\oti L$. We can therefore say that $\rho'(W)\subset W\otimes_K(L\oti K)$, which proves that the functor   $(\varphi\oti K)_*$ preserves subobjects. 
 
Let us now verify that $(\varphi\oti K)_*$
is fully faithful. Now, faithfulness is obvious and so let us give ourselves 
$V$ and $W$ in $\Comodf{L\oti K}$ and $h:V\to W$ a map of $L'\oti K$-comodules. 
Because of Lemma \ref{lem_serre},  
we may assume  $V=\mathcal V\oti_RK$
 and $W=\mathcal W\oti_RK$ for certain 
$\mathcal V$ and $\mathcal W$ in $\Comodf{L}$. Clearing denominators, there exists $d\in R$  such that $dh(\mathcal V)\subset \mathcal W$. Little effort is required to verify that $dh$ defines a morphism of $L'$-comodules, and hence is also a morphism of $L$-comodules by our assumption that $\varphi_*:\Comodf L\to \Comodf{L'}$ is full. It then follows that $(\varphi\oti K)_*$ is full. 
\end{proof}

\begin{lem}\label{lem_serre}
Let $R$ be a noetherian domain with quotient field $K$, and $L$ be an $R$-flat coalgebra. Then any $L\oti K$-comodule of finite dimension has a model over $R$, i.e. an $R$-module $\mathcal V$ of finite type affording a coaction of $L$ such that $\mathcal V\oti_RK\simeq V$ as $L\oti K$-comodules.
\end{lem}
\begin{proof}
This   is a consequence of  \cite[Prop.~2, page 40]{Se68}, as we now exlpain. Let $\{v_i\}$ be a $K$-basis of $V$; since $V\oti_K(L\oti K)=V\oti_RL$, we may  consider $V$ as an $L$-comodule. Then there exists a 
subcomodule $\mathcal V$ of $V$, finite over $R$ and containing $\{v_i\}$. This is the required  model.  
\end{proof}

\begin{rmk}
Proposition \ref{flat_criterion_coalgebra_hom} is an enhancement of the main result of \cite[Prop.~2.7]{Hai16}. In the latter proposition, the whole category of comodules is used, not just those of finite type over $R$. This provides an evidence that the might exist a Tannakian formalism over any noetherian ring. 
\end{rmk}

\section*{Acknowledgments}The first named author (PHH) was supported by the Vietnam Academy of Science and Technology under grant number CTTH00.02/23-24. He was also  funded by Vingroup Joint Stock Company and  supported by Vingroup Innovation Foundation (VinIF) under the project code VINIF.2021.DA00030. HDN was partially supported by the Vietnam Academy of Science and Technology (through the grant NCXS02.01/22-23).


\begin{thebibliography}{99}

\bibitem[Ba+23]{barthel-et-al23} T.  Barthel, D.  Benson, S.  Iyengar, H. Krause, and J. Pevtsova,
 {\it Lattices over finite group schemes and stratification}. Preprint 	arXiv:2307.16271. 
 
\bibitem[BA]{bourbaki_algebra}N. Bourbaki, {\it Elements of Mathematics. Algebra I}, Springer (1989).



\bibitem[CIM19]{CIM}
L.W. Christensen, S.B. Iyengar, and T. Marley,
\emph{Rigidity of Ext and Tor with coefficients in residue fields of a commutative Noetherian ring}.
Proc. Edinb. Math. Soc. (2) {\bf 62} (2019), no. 2, 305--321. 


\bibitem[DM82] {DM82}   P. Deligne, J. Milne, \textit{ Tannakian Categories}, Lecture Notes in Mathematics {\bf 900}, p. 101-228, Springer Verlag (1982). 

 

\bibitem[DH18]{DH18} N.D. Duong and  P.H. Hai, \emph{Tannakian duality over Dedekind rings and applications}. Math. Z. (2018) {\bf 288}:1103--1142.



\bibitem[EGA]{ega}   A. Grothendieck (with the collaboration of J. Dieudonn\'e). \emph{\'El\'ements de G\'eom\'etrie Alg\'ebrique}. Publ. Math. IH\'ES \textbf{8}, \textbf{11} (1961); \textbf{17} (1963);  \textbf{20} (1964); \textbf{24} (1965); \textbf{28} (1966); \textbf{32} (1967). 
  

\bibitem[Hai16]{Hai16} P.H. Hai, \emph{On an injectivity lemma in the proof of Tannakian duality}, J. Alg. and Its  App., (2016), Vol. {\bf 15}, No. 9 (2016) 1650167 (9 pages).


  \bibitem[Ja03]{jantzen} J.C. Jantzen, \emph{Representations of Algebraic Groups}. Second edition. Mathematical Surveys and Monographs, {\bf 107}. American Mathematical Society, Providence, RI, 2003.


\bibitem[Mac98]{maclane98}
S. Mac Lane, \emph{Categories for the Working Mathematician}, Graduate Texts in Mathematics {\bf 5}, Second edition, Springer-Verlag, New York, 1998. 



\bibitem[Mat89]{Mat89} H. Matsumura, \textit{Commutative Ring Theory}, Cambridge Studies in Advanced Mathematics
8, Cambridge University Press (1989).


\bibitem[Ni02]{nitsure}N. Nitsure, 
{\it Representability of $\mathbf{GL}_E$}. Proc. Indian Acad. Sci. (Math. Sci.) {\bf 112}(4), November 2002,  539--542.

    
\bibitem[Se68]{Se68} J.-P. Serre,  \textit{Groupe de Grothendieck des sch\'emas en groupes r\'eductifs d\'eploy\'es}, Publ. Math. IH\'ES {\bf 34} (1968), p.37-52.

   
\bibitem[Wat79]{Wat}  W.C. Waterhouse, \textit{ Introduction to Affine Group Schemes}, Springer-Verlag (1979). 
\end{thebibliography}
\end{document}